\documentclass[12pt, a4paper]{amsart}
\usepackage[utf8]{inputenc}
\usepackage[T1]{fontenc}
\usepackage[english]{babel} 
\usepackage{amsmath, amssymb, amsthm, amsfonts}
\usepackage{mathtools}
\usepackage{mathrsfs}
\usepackage{geometry}
\usepackage{xcolor}
\usepackage{hyperref}
\usepackage[toc]{appendix}
\newtheorem{theorem}{Theorem}[section]
\newtheorem{lemma}[theorem]{Lemma}
\newtheorem{corollary}[theorem]{Corollary}
\newtheorem{proposition}[theorem]{Proposition}
\newtheorem{definition}[theorem]{Definition}

\newtheorem{remark}[theorem]{Remark}
\newtheorem{notation}[theorem]{Notation}
\newtheorem{example}[theorem]{Example}
\newtheorem*{problem}{Open problem}
\newcommand{\Weyl}{A_n(\mathbb{K})} 
\newcommand{\Invar}{A_n^G(\mathbb{K})}           
\newcommand{\GKdim}{\text{GKdim}}    
\def\k{\mathbb K}
\def\C{\mathbb C}

\newcommand{\Ann}{{\rm{Ann}}}
\newcommand{\Frac}{{\rm{Frac}}}

\newcommand{\Spec}{{\rm{Spec}}}

\newcommand{\Autk}{{\rm{Aut}}}

\newcommand{\Der}{{\rm{Der}}}

\newcommand{\GL}{{\rm{GL}}}
\def\k{\mathbb K}

\def\D{\mathcal D}

\def\Z{\mathbb Z}

\def\g{\mathfrak g}
\def\h{\mathfrak h}

\def\KK{\mathbb K}

\def\G{\mathcal G}
\def\c{\mathfrak c}

\def\Gskew{\mathcal{F}}
\def\Ginv{\mathcal{F}^*}
\def\H{\mathcal H}

\newcommand{\Mod}{\mathsf{Mod}}

\title{Irreducible non-holonomic modules for rational Cherednik algebras}
\author{Felipe Albino dos Santos \and João Fernando Schwarz}
\begin{document}
\begin{abstract}
Let $\mathbb{K}$ be an algebraically closed field of characteristic zero and let
$A_n(\mathbb{K})$ be the $n$-th Weyl algebra. We prove that for every complex
reflection group $G$ and every $n \geq 2$ the ring of invariant differential
operators $A_n(\mathbb{K})^G$ has irreducible non-holonomic modules of
Gelfand-Kirillov dimension $2n-1$, and we exhibit such modules explicitly.
Through the Morita equivalence between $A_n(\mathbb{K})^G$ and the rational
Cherednik algebra $H_{\mathfrak c}$ at an integral parameter, we deduce that
$H_{\mathfrak c}$ has irreducible non-holonomic modules of Gelfand-Kirillov
dimension $2n-1$ for $n \geq 2$. In an appendix, following a proof kindly shared by O. Mathieu, we show
that $A_n(\mathbb{K})$ has irreducible modules of every Gelfand-Kirillov
dimension in the interval $[n, 2n-1]$.
\end{abstract}

\subjclass[2020]{Primary 16S32,16P90; Secondary 16D60, 16S35, 20F55,16W22}
\keywords{Weyl algebra, rings of invariant differential operators,
non-holonomic modules, Gelfand-Kirillov dimension, rational Cherednik algebras,
complex reflection groups}

\maketitle
\section{Introduction}
The notion of holonomic modules for the Weyl algebra $\Weyl$ was first introduced by I. Bernstein \cite{INBernstein}, who used it in an elementary way to obtain a new proof of a problem posed by I. M. Gelfand at the International Congress of Mathematicians in Amsterdam, 1954 (see \cite[Chapter 8]{KL} for the full story). The notion is related to maximally overdetermined systems of partial differential equations (cf. \cite{Sato}).
The \textit{Bernstein inequality} says that, for every finitely generated module $M$ over $\Weyl$, we have the following inequality for the Gelfand-Kirillov dimension: $\GKdim \, M \geq n$ \footnote{It is also bounded above by $\GKdim \, \Weyl=2n$ \cite[Chapter 5]{KL}.}. Such a module is called \textit{holonomic} precisely when the dimension attains the minimal value $n$.
We recall that the Weyl algebra $\Weyl$ can also be realized as the ring of differential operators on the affine space $\KK^n$. In general, if $X$ is a smooth affine variety\footnote{All of our varieties will be irreducible.} and $M$ is a finitely generated module for $\D(X)$, the algebra of differential operators on $X$, then $\GKdim \, M \geq \dim \, X$, and the module is again called holonomic when the Gelfand-Kirillov dimension is minimal (for a proof, see \cite[Theorem 4.3]{McConnell}).
We have a similar inequality for Lie algebras, called \textit{Gabber's inequality}: let $\mathfrak{g}$ be an algebraic Lie algebra over an algebraically closed field of characteristic $0$, and $M$ a finitely generated $U(\g)$-module. Then
    \[
        \GKdim \, M \geq \frac{1}{2} \GKdim \big(U(\mathfrak{g})/\Ann(M)\big).
    \]
We refer to \cite[Theorem 9.11]{KL} for a proof. Equality holds for modules in the category $\mathcal{O}$ and Harish-Chandra modules (e.g., \cite{Jantzen}). If $\g$ is not algebraic, it is known that Gabber's inequality can fail \cite{McConnell3}.
For rings of differential operators, another way to define holonomic modules is by using their characteristic variety (\cite{Hotta}). This approach was used by Losev in \cite{Losev} to introduce the notion of holonomic modules for many algebras.
In particular, his work applies to rational Cherednik algebras. These were introduced in the foundational work of P. Etingof and V. Ginzburg on symplectic reflection algebras \cite{EG}, of which rational Cherednik algebras are one of the most prominent families. They can be seen as a two-step degeneration of Cherednik's double affine Hecke algebra (DAHA).
Yet another way to understand rational Cherednik algebras is via deformation theory. They form a universal flat family of deformations of the skew group ring $A_n(\C)*G$ of the complex Weyl algebra $A_n(\C)$ and a complex reflection group $G$. When the deformation parameter $\mathfrak{c}$ is regular, one can define holonomic modules using the Gelfand-Kirillov dimension as well \cite[Proposition 3.4]{Thompson}. Holonomic modules for rational Cherednik algebras were further studied in \cite{Bellamy}. In particular, every module in category $\mathcal{O}$ \cite{GGOR} is holonomic.
For some time, it was believed, mainly due to a lack of counterexamples, that every irreducible module over the Weyl algebra is holonomic \cite[Problem 1, Chapter 1, \S~9]{Bjork}. The first examples of non-holonomic irreducible modules for $\Weyl$ (of Gelfand-Kirillov dimension $2n-1$), as well as for $U(\mathfrak{sl}_2 \times \mathfrak{sl}_2)$, appeared in the work of T. Stafford \cite{Stafford}. Later, Bernstein and Lunts \cite{BerL, Lunts} showed that for generic $d \in \Weyl$, the module $\Weyl/\Weyl d$ is irreducible non-holonomic. Irreducible non-holonomic modules for rings of differential operators were further studied in \cite{Coutinho0}, \cite{Coutinho03}, \cite{Coutinho06}, and \cite{Coutinho07}. In \cite{Coutinho17} it was shown that if $\mathfrak{g}$ is any semisimple complex Lie algebra other than $\mathfrak{sl}_2$, then $U(\g)$ has simple non-holonomic modules\footnote{The existence of irreducible non-holonomic modules for $\mathfrak{sl}_2$ remains an open problem.}.
The Bernstein inequality also holds for fixed rings $\Invar$ under the action of \textit{any} finite group. This was shown in \cite{FS4} when $G$ is a group of symplectomorphisms of the Weyl algebra, by computing the value of the \textit{filter dimension} \cite{Bavula2} for $\Invar$. Again, we call a finitely generated $\Invar$-module $M$ holonomic if $\GKdim \, M= n$.
The first main result of this paper is that, if $G$ is a complex reflection group, then $\Invar$ has irreducible non-holonomic modules of Gelfand-Kirillov dimension $2n-1$ (Theorem \ref{irred-non-hol-invar-Weyl}). The method of our proof allows us to exhibit such modules explicitly. This result is then used to prove our second main result: when the parameter $\c$ of the rational Cherednik algebra is integral (cf. \cite{Berest}), irreducible non-holonomic modules exist (Theorem \ref{iredd-non-hol-Cherednik}).
The second section of this paper discusses preliminaries on the Gelfand-Kirillov dimension, noncommutative invariant theory, and the relevant results of \cite{Coutinho0} that will be used in the sequel.
The third section discusses some basic properties of $\Invar$-modules, where $G$ is \textit{any} group of automorphisms. The fourth section is dedicated to the proof of Theorem \ref{irred-non-hol-invar-Weyl}, and the fifth section discusses the case of rational Cherednik algebras.
In an Appendix, we include the proof that $\Weyl$ has irreducible modules of \textit{any} Gelfand-Kirillov dimension in the interval $[n,2n-1]$, which seems to be well known to specialists but for which we could not find any reference. We thank Olivier Mathieu for allowing us to include his simple proof of this fact.
\section{Preliminaries}

\subsection{Gelfand-Kirillov dimension}
The Gelfand-Kirillov dimension is an important dimension in ring theory and is the basis of our ring-theoretic approach to the notion of holonomic modules, as discussed in the Introduction. Canonical references are \cite{KL}, \cite{Lorenz0}, and \cite[Chapter 8]{McConnell}.
\begin{definition}
Let $\Phi$ be the set of functions $f: \mathbb{N} \rightarrow \mathbb{R}$ such that there exists an $n_0 \in \mathbb{N}$ with:
\[f(n)>0\text{ and } f(n+1) \geq f(n), \quad\forall \, n \geq n_0.\]
For two functions in $\Phi$, we write $f \leq^* g$ if there are constants $c, m \in \mathbb{N}$ such that $f(n) \leq cg(mn)$. If $g \leq^* f$ as well, we write $f \sim g$. We represent the class of $f$ in $\Phi/\sim$ by $\mathcal{G}(f)$. This equivalence class of $f$ is called its \emph{growth}, and the relation induced by $ \leq^* $ in $\Phi/\sim$ is denoted $\leq$.
\end{definition}
\begin{notation}
Let $y \in \mathbb{R}, y \geq 0$. The growth of the function $n \mapsto n^y$ is denoted $\mathcal{P}_y$.
Let $y \in \mathbb{R}, y>0$. The growth of the function $n \mapsto e^{n^y}$ is $\mathcal{E}_y$.
\end{notation}
If $f(n)=\ln(n+1)$, then $\mathcal{G}(f) > \mathcal{P}_0$ but $\mathcal{G}(f) < \mathcal{P}_\varepsilon$ for any $\varepsilon >0$. So there is no analogue of the Archimedean property for this order.
Finally, we remark that not every pair $\mathcal{G}(f)$, $\mathcal{G}(h)$ is comparable.
\begin{notation}\label{asymptotic}
For $f \in \Phi$,
\[\gamma(f)= \limsup_{n \to \infty}\big(\log_n \, f(n)\big).\]
\end{notation}
\begin{proposition}\label{growth}
    Let $f,g \in \Phi$.
    \begin{enumerate}
        \item If $\mathcal{G}(f)=\mathcal{G}(g)$, then $\gamma(f)=\gamma(g)$.
        \item $\gamma(f+g)=\sup \{ \gamma(f), \gamma(g) \}$.
        \item If $f(n)=p(n)$ for $n\gg 0$ and $p(x) \in \mathbb{R}[x]$, then $\gamma(f)=\deg \, p$.
        \item If $\gamma(f)=\lim_{n \to \infty} \log_n \, f(n)$, then $\gamma(fg)= \gamma(f)+\gamma(g)$.
        \item If $g(n) \leq f(an+b), \, n\gg 0$, where $a,b \in \mathbb{N}$, then $\gamma(g) \leq \gamma(f)$.
    \end{enumerate}
\end{proposition}
\begin{proof}
    \cite[Lemma 2.1b)]{KL}, \cite[Lemma 8.1.7]{McConnell}.
\end{proof}
\begin{definition}
Let $A$ be an affine algebra and $V$ a finite-dimensional generating space for $A$. Let $d_V(n)=\dim \, \sum_{i=0}^n V^i$. Then $\mathcal{G}(d_V)$ is independent of the choice of the generating space $V$ (\cite[Lemma 1.1]{KL}). Hence,
\[ \gamma(d_V)=\limsup_{n \rightarrow \infty} \big(\log_n(d_V(n))\big),\]
is also independent of $V$, so by Proposition \ref{growth}(1), this number is well-defined and is called the \emph{Gelfand-Kirillov dimension} of $A$, denoted $\GKdim \, A$. We also put $\mathcal{G}(A)=\mathcal{G}(d_V)$, where $V$ is any choice of finite-dimensional generating space for $A$, and call it \emph{the growth of} $A$.
\end{definition}
\begin{remark}
When $1 \in V$, $d_V(n)$ reduces to $\dim \, V^n$, and $V$ is called a frame.
\end{remark}
\begin{example}
The Gelfand-Kirillov dimension of the Weyl algebra $\Weyl$ over a field $\k$ of zero characteristic is $2n$ \cite[Proposition 8.1.15(ii)]{McConnell}.
\end{example}
\begin{definition}
    Let $M$ be a finitely generated left $A$-module, where $A$ is an affine algebra. Choose a finite-dimensional space $F \subset M$ that generates it as a module, and let $V$ be a frame for $A$. Let $d_{F,V}(n)=\dim \, V^n F$. The \emph{Gelfand-Kirillov dimension} of $M$ is
    \[ \GKdim \, M=\GKdim_A\, M=\limsup_{n \to \infty} \, \log_n \, d_{F,V}(n), \]
    and the value is independent of the choices of $F, V$ by \cite[Proof of Lemma 1.1]{KL}.
\end{definition}
\begin{example}
Let $M:=\k[x_1,\ldots,x_n]$ be the polynomial module for $\Weyl$. Then its Gelfand-Kirillov dimension is $n$ (e.g., \cite[Theorem 8.5.10]{McConnell}).
\end{example}
We recall here two results about the Gelfand-Kirillov dimension that will be important for us.
\begin{proposition}\label{basic-GK-1}
    Let $\theta: S \rightarrow R$ be a homomorphism of algebras, which allows us to regard any left $R$-module $M$ as a left $S$-module. Then $\GKdim_S \, M \leq \GKdim_R \,  M$, with equality $\GKdim_S \, M = \GKdim_R \,  M$ when $R$ is a finitely generated $S$-module via $\theta$. \end{proposition}
\begin{proof}
    \cite[Proposition 1.6(d)]{Lorenz0}.
\end{proof}
In general, the Gelfand-Kirillov dimension behaves poorly with respect to localization \cite[Chapter 4]{KL}. For instance, while $\GKdim \, \Weyl=2n$, the Gelfand-Kirillov dimension of the Weyl field $\Frac \, \Weyl$ is $\infty$ (e.g., \cite[Theorem 8.17]{KL}). The following is one of the few known cases in which localization behaves well.
\begin{theorem}\label{basic-GK-2}
    Let $A$ be an affine $\k$-algebra and $S$ a multiplicatively closed subset of $A$ consisting of regular elements. Assume $S$ is commutative and, for each $s \in S$, $\operatorname{ad}(s)$ is a locally nilpotent derivation of $A$. Assume as well that $A$ has no $S$-torsion. Then $S$ is an Ore set in $A$, and for each $A$-module $M$, $\GKdim_A M= \GKdim_{AS^{-1}} MS^{-1}$.
\end{theorem}
\begin{proof}
    \cite[Theorem 3.2]{Lorenz0}.
\end{proof}
\subsection{Noncommutative invariant theory}
Let $G$ be a finite subgroup of $\Autk_\KK \, R$, where $R$ is a $\KK$-algebra. We have the \emph{trace function} $\tau:R\rightarrow R^G$, where $\tau(r)=\sum_{g \in G} r^g, \, r \in R$.
We begin by recording a simple lemma.
\begin{lemma}\label{most-basic-lemma}
Let $R$ be a $\KK$-algebra and $G \subset \Autk_\KK R$ a finite group such that $|G|^{-1} \in \KK$. Let $e=1 /|G| \sum_{g \in G} g$. Denoting the skew group ring as $R*G$, we have that $e^2=e$, $e(R*G)=eR$, $e(R*G)e=eR^G\simeq R^G$ and $ere=e\tau(r/|G|)$ for $r \in R$. More generally, if $\Lambda \subset R$ is a subalgebra with $G(\Lambda) \subset \Lambda$, then $\tau(\Lambda)=\Lambda^G$ and $e \Lambda e=e \Lambda^G \simeq \Lambda^G$.
\end{lemma}
\begin{proof}
    For the first half, \cite[Lemma 2.1 and its proof]{Montgomery}, with isomorphism $\psi: R^G \simeq eR^G$ given by $ r \mapsto er$, $ r \in R^G$, as $eR^G=R^Ge$. The second half follows from the first: $e \Lambda e= e \tau(\Lambda)= e \Lambda^G$, by hypothesis, and the restriction of $\psi$ to $ \Lambda^G$ gives us the desired result.
\end{proof}
\begin{theorem}\label{MS}
Let $R$ be a finitely generated Noetherian $\k$-algebra. If $G$ is a finite group of automorphisms of $R$, with $|G|^{-1} \in \k$, then $R^G$ is a finitely generated $\k$-algebra, $R^G$ is Noetherian, and $R$ is a finitely generated $R^G$-module.
\end{theorem}
\begin{proof}
    The first claim is proven in \cite{MS}, the second one in \cite[Corollary 1.12]{Montgomery}, and the third one in \cite[Corollary 5.9]{Montgomery}.
\end{proof}
When our algebra is simple, more can be said:
\begin{theorem}\label{essential-invariant}
    Let $R$ be a simple $\k$-algebra, and $G$ a finite group of outer automorphisms of $R$ such that $|G|^{-1} \in \k$. Then
    \begin{enumerate}
        \item $R^G$ is a simple ring.
        \item $R$ is a finitely generated projective $R^G$-module.
        \item $R^G$ and $R*G$ are Morita equivalent. In particular, if $R$ is Noetherian, $R^G$ is Noetherian.
    \end{enumerate}
\end{theorem}
\begin{proof}
    \cite[Theorem 2.4, Theorem 2.5, Corollary 2.6]{Montgomery}.
\end{proof}
This is relevant to the study of $\Weyl$:
\begin{proposition}\label{Weyl-outer}
    Every automorphism of the Weyl algebra is outer.
\end{proposition}
\begin{proof}
    \cite[2.4.1]{Dumas}.
\end{proof}
\begin{corollary}\label{Morita-invariants-Weyl}
   $\Weyl *G$ and $\Invar$ are Morita equivalent algebras. Hence the latter is a finitely generated, Noetherian simple domain with Gelfand-Kirillov dimension $2n$.
\end{corollary}
\begin{proof}
    The Gelfand-Kirillov dimension of an algebra is a Morita invariant \cite[Proposition 8.2.9(iii)]{McConnell}.
\end{proof}
We need another important theorem about fixed rings of algebras of differential operators under the action of a finite group. If $X$ is a smooth affine variety over an algebraically closed field $\KK$ of zero characteristic, and $G$ is a finite group that acts on $X$, then $G$ acts on the algebra of differential operators $\D(X)$ by conjugation: $(g.D)(f)=(g \circ D \circ g^{-1})(f)$, $g \in G$, $D \in \D(X)$, $f \in \KK[X]$.
\begin{theorem}\label{Knop}
    Let $X$ be a smooth affine variety over an algebraically closed field of zero characteristic, and $G$ a group that acts on $X$, and hence, as above, on $\D(X)$. Then $\D(X)^G=\{D \in \D(X) \mid D(\KK[X]^G) \subset \KK[X]^G \}.$
\end{theorem}
\begin{proof}
    \cite[Theorem 3.1]{Knop}.
\end{proof}
The following theorem is probably well known to specialists, but its proof is difficult to find in the literature.
\begin{theorem}\label{GK-is-Morita-for-modules}
    Let $A$ and $B$ be two Morita equivalent algebras and $F: A\text{-}\Mod \rightarrow B\text{-}\Mod$, $G: B\text{-}\Mod \rightarrow A\text{-}\Mod$ the two functors that realize the equivalence of the categories of modules. Let $M$ be a finitely generated $A$-module. Then $\GKdim_A \, M = \GKdim_B \, F(M)$.
\end{theorem}
\begin{proof}
    \cite[Theorem 1.7]{BH}.
\end{proof}
\subsection{D-simple rings and D-modules on select varieties}
\begin{definition}
    Let $A$ be a commutative $\KK$-algebra, and let $d \in \Der_\KK A$ be a derivation. $A$ is called a $d$-\textit{simple algebra} if it does not contain nonzero proper ideals $I$ such that $d(I) \subset I$.
\end{definition}
When $A=\KK[x_1,\ldots,x_n]$ is the polynomial algebra, there are many derivations that make it $d$-simple. Denote $\KK[x_1,\ldots,x_n]$ by $\KK[\bar{X}]$. Following an idea of Shamsuddin (cf. \cite[Theorem 2.3.16]{Archer}), we have:
\begin{theorem}\label{d-simple}
    Let $a_2, \ldots, a_n$, $b_2, \ldots, b_n$ be nonzero polynomials of $\KK[x_1]$. If $a_i/a_j \notin \mathbb{Q}$ for $2 \leq i < j \leq n$ and $\deg a_i > \deg b_i$ for $i=2, \ldots ,n$, then $\KK[\bar{X}]$ is $d$-simple, where
    \[ d= \partial_1 + \sum_{i \geq 2}(x_i a_i + b_i) \partial_i.\]
\end{theorem}
\begin{proof}
    \cite[Theorem 3.3]{Coutinho0}.
\end{proof}
In \cite{Stafford}, a family of derivations $d$ of $\KK[\bar{X}]$ was given, such that $\Weyl/\Weyl d$ is irreducible non-holonomic of Gelfand-Kirillov dimension $2n-1$. We will now quote another result from \cite{Coutinho0} that generalizes Stafford's method:
\begin{theorem}\label{Stafford-generalized}
    For $2 \leq i \leq n$, let $a_i, b_i, h_i$ be nonzero polynomials of $\KK[\bar{X}]$ such that the $a_i$ are linearly independent over $\mathbb{Q}$ and $\deg \, a_i > \max\{\deg b_i, \deg h_i \}$ for $i=2,\ldots, n$. If $d$ is as in Theorem \ref{d-simple} and $h=\sum_{i \geq 2} h_i x_i$, then $\Weyl/\Weyl(d+h)$ is an irreducible non-holonomic module of Gelfand-Kirillov dimension $2n-1$.
\end{theorem}
\begin{proof}
    \cite[Theorem 3.6]{Coutinho0} and the fact that $d+h$ is an operator of degree greater than or equal to one (\cite[pp. 407]{Coutinho0}).
\end{proof}
Now assume $\KK$ is algebraically closed.
\begin{definition}
    A smooth affine variety $X$ of dimension $n$ is called a \textit{select} variety if its module of K\"ahler differentials $\Omega_\KK X$ is free of rank $n$ and has a basis $dx_1, \ldots, dx_n$.
\end{definition}
Let $x_1, \ldots, x_n$ be the corresponding elements of $\mathcal{O}(X)$, which are algebraically independent. Let $\KK[\widehat{X}]$ be the polynomial algebra $\KK[x_1,\ldots,x_n]$. Then there are derivations $\partial_1, \ldots, \partial_n$ of $\mathcal{O}(X)$ such that $\partial_i(x_j)=\delta_{ij}$, and moreover $\Der_\KK \mathcal{O}(X)$ is a free $\mathcal{O}(X)$-module with basis $\partial_1, \ldots, \partial_n$ \cite[Corollary 15.1.12]{McConnell}. So in particular, every derivation $d$ of $\KK[\widehat{X}]$ extends to a unique derivation of $\mathcal{O}(X)$. It is also clear that $\D(X)$ contains a copy of the Weyl algebra $\Weyl=\D(\KK[\widehat{X}])$.
\begin{theorem}\label{theorem-select-varieties}
    Let $X$ be a select variety and let $d$ be a derivation of $\KK[\widehat{X}]$ such that $\KK[\widehat{X}]$ is $d$-simple. Then
    \begin{enumerate}
    \item $\mathcal{O}(X)$ is $d$-simple.
    \item if, for some $f \in \KK[\widehat{X}]$, $\Weyl/\Weyl(d+f)$ is a simple $\Weyl$-module, then the $\D(X)$-module $\D(X)/\D(X)(d+f)$ is also simple. Its Gelfand-Kirillov dimension is $2n-1$.
    \end{enumerate}
\end{theorem}
\begin{proof}
    \cite[Theorem 2.2]{Coutinho0} and the fact that $(d+f)$ is a differential operator of degree $\geq 1$.
\end{proof}
\section{Bernstein inequality for $\Invar$ and basic properties}
In this section, for completeness, we give an elementary proof of the Bernstein inequality for the algebras $\Invar$ and discuss some basic results about the Gelfand-Kirillov dimension of their modules.
\begin{proposition}\label{Bernstein-skew-product-ring}
    Let $G$ be \textit{any} (e.g., non-linearizable) finite subgroup of automorphisms of $\Weyl$. Then, if $M$ is any finitely generated $\Weyl*G$-module, $\GKdim \, M \geq n$. Modules with minimal Gelfand-Kirillov dimension exist.
\end{proposition}
\begin{proof}
    If $M$ is a finitely generated module for $\Weyl*G$, in particular it is one for $\Weyl$, hence by the usual Bernstein inequality and Proposition \ref{basic-GK-1}, $\GKdim_{\Weyl*G} \, M = \GKdim_{\Weyl} M \geq n$. Let $\KK[\bar{X}]$ be the polynomial algebra in $n$ variables, with $\Weyl$ acting as usual as differential operators with polynomial coefficients, and $G$ acting trivially. Then, again by Proposition \ref{basic-GK-1}, $\GKdim_{\Weyl*G} \KK[\bar{X}]=\GKdim_{\Weyl} \KK[\bar{X}]=n$.
\end{proof}
\begin{proposition}\label{Bernstein-invariants}
Let $G$ be any finite subgroup of automorphisms of $\Weyl$. If $M$ is any finitely generated $\Invar$-module, $\GKdim \, M \geq n$, and modules with minimal Gelfand-Kirillov dimension exist.
\end{proposition}
\begin{proof}
    Use the previous Proposition \ref{Bernstein-skew-product-ring}, together with Corollary \ref{Morita-invariants-Weyl} and Theorem \ref{GK-is-Morita-for-modules}.
\end{proof}
\begin{definition}
    Let $A$ be $\Weyl*G$ or $\Invar$. If an $A$-module $M$ has Gelfand-Kirillov dimension $n$, it is called holonomic.
\end{definition}
\begin{example}
    As a more interesting example, we will show that $\k[\bar{X}]^G$ is an irreducible holonomic $\Invar$-module. It is indeed an $\Invar$-module (Theorem \ref{Knop}). Moreover, $\k[\bar{X}]^G$ is an $\Invar$-submodule of $\k[\bar{X}]$, which has Gelfand-Kirillov dimension $n$ by Theorem \ref{essential-invariant} together with Proposition \ref{basic-GK-1}. Since $\k[\bar{X}]^G$ is a nonzero $\Invar$-submodule, by Proposition \ref{Bernstein-invariants} its Gelfand-Kirillov dimension is also $n$. It is an irreducible module, again by Theorem \ref{Knop}.
\end{example}
\begin{definition}
    Let $\G(j)$, where $j \in [n, 2n]$, be the set of isomorphism classes of $\Weyl$-modules with Gelfand-Kirillov dimension $j$.
\end{definition}
By \cite[Exercise 9.5.3]{Coutinho}, $\G(j) \neq \emptyset$ for every $j \in [n, 2n]$. Denote by $\Ginv(j)$ the set of isomorphism classes of $\Invar$-modules with Gelfand-Kirillov dimension $j$, and $\Gskew(j)$ the set of isomorphism classes of $\Weyl*G$-modules with Gelfand-Kirillov dimension $j$.
\begin{proposition}\label{g-skew}
    If $M \in \G(j)$, then $M^*=\Weyl*G \otimes_{\Weyl} M$ belongs to $\Gskew(j)$. If $M \in \Gskew(j)$, then $M|_{\Weyl} \in \G(j)$.
\end{proposition}
\begin{proof}
    Immediate from Proposition \ref{basic-GK-1}.
\end{proof}
\begin{proposition}\label{g-inv}
    If $M \in \G(j)$, then $M|_{\Invar} $ belongs to $\Ginv(j)$. If $M \in \Ginv(j)$, then $\Weyl \otimes_{\Invar} M$ belongs to $\G(j)$.
\end{proposition}
\begin{proof}
    Immediate from Proposition \ref{basic-GK-1} again.
\end{proof}
So for every $j \in [n,2n]$, $\Ginv(j)$ and $\Gskew(j)$ are non-empty.
\begin{definition}
    We denote by $\H*G$, $\H$ and $\H^G$ the categories of holonomic modules (plus the zero module) for $\Weyl*G$, $\Weyl$, and $\Invar$, respectively.
\end{definition}
\begin{proposition}
 Let $G$ be any finite group of automorphisms of $\Weyl$. The functor $\operatorname{Ind}^G: \mathcal{H}^G \rightarrow \mathcal{H}$, which sends $M$ to $\Weyl \otimes_{\Invar} \, M$, is left adjoint to the functor $\operatorname{Res}_G: \mathcal{H} \rightarrow \mathcal{H}^G$ given by restriction of scalars from $\Weyl$ to $\Invar$. Both $\operatorname{Ind}^G$ and $\operatorname{Res}_G$ are exact functors.
\end{proposition}
\begin{proof}
    $\operatorname{Ind}^G$ is exact because $\Weyl$ is a projective, hence flat, $\Invar$-module (Theorem \ref{essential-invariant}). The exactness of $\operatorname{Res}_G$ is obvious. The categories of holonomic modules are preserved by Proposition \ref{g-inv}.
\end{proof}
\begin{theorem}
    The Morita equivalence between $\Weyl*G$ and $\Invar$ induces an equivalence between the abelian categories $\H*G$ and $\H^G$.
\end{theorem}
\begin{proof}
    An immediate consequence of Corollary \ref{Morita-invariants-Weyl}.
\end{proof}
The following theorem summarizes the main results of \cite{FS4} on the invariants of the Weyl algebra $\Invar$, for $G < \GL_n(\k)$.
\begin{theorem}\label{main-properties-holonomic-modules}
    \begin{enumerate}
\item $\H^G$ is an abelian, extension-closed subcategory of $\Invar\text{-}\operatorname{Mod}$.
\item Every holonomic module for $\Invar$ is cyclic.
\item Every holonomic module for $\Invar$ is torsion.
\item If $n=1$, every irreducible module for $A_1(\KK)^G$ is holonomic, and a module is holonomic if and only if it is torsion.
\end{enumerate}
\end{theorem}
\section{Irreducible non-holonomic modules for fixed rings of the Weyl algebra}
In this section, $\KK$ will be an algebraically closed field of zero characteristic, and we consider modules for $\Invar$, where $G$ is a complex reflection group acting through its natural representation. Our aim is to find irreducible non-holonomic representations for $\Invar$ of Gelfand-Kirillov dimension $2n-1$.
By the Chevalley-Shephard-Todd theorem, $\KK[x_1, \ldots, x_n]^G=\KK[e_1,\ldots,e_n]$, where the $e_i$ are algebraically independent elements. We write $\KK[\bar{E}]$ for $\KK[e_1,\ldots,e_n]$.
\begin{theorem}\label{EFS}
    Let $G$ be any complex reflection group. Then there is a $G$-invariant polynomial $\Delta \in \KK[\bar{X}]$ such that, for some polynomial $f \in \KK[\bar{E}]$,
    \[ \D(\KK[\bar{X}])^G_\Delta \simeq \D(\KK[\bar{E}])_f \simeq \mathcal{D}(\KK[\bar{E}]_f), \] where $f$ is just $\Delta$ expressed as a polynomial in the $e_i$. All these isomorphisms preserve the filtration by order of differential operator.
\end{theorem}
\begin{proof}
    \cite[Proposition 5.1, 5.2, and 5.7]{FS}; \cite[Proof of Theorem 2]{EFOS}.
\end{proof}
Assume $n \geq 2$ (cf. Theorem \ref{main-properties-holonomic-modules}(4)).
Let $X=\Spec \,\KK[\bar{E}]$, $X_f=\Spec \, \KK[\bar{E}]_f$, and $X_\Delta = \Spec \, \KK[\bar{X}]_\Delta$, where $\Delta$ and $f$ are as in Theorem \ref{EFS} above.
By \cite[15.1.24]{McConnell}, $\Omega_\KK \,\mathcal{O}(X_f)$ is a free $\mathcal{O}(X_f)$-module with basis $de_1, \ldots, de_n$, and similarly $\operatorname{Der}_\KK \, \mathcal{O}(X_f)$ is a free module with basis $\partial_{e_1}, \ldots, \partial_{e_n}$. Hence, $X_f$ is a select variety. Clearly $\KK[\bar{E}] \subset \mathcal{O}(X_f)$, and $\partial_{e_i}(e_j)=\delta_{ij}$.
Let us pick a derivation (Theorem \ref{d-simple}) $d$ such that $\KK[\bar{E}]$ is a $d$-simple algebra, and $h \in \KK[\bar{E}]$ such that $\D(\KK[\bar{E}])/\D(\KK[\bar{E}])(d+h)$ is an irreducible $\D(\KK[\bar{E}])$-module of Gelfand-Kirillov dimension $2n-1$ (Theorem \ref{Stafford-generalized}).
Then $\D(X_f)/\D(X_f)(d+h)$ is a simple $\D(X_f)$-module of Gelfand-Kirillov dimension $2n-1$ by Theorem \ref{theorem-select-varieties}.
From now on $\Invar$ will denote $\D(X)^G$.
\begin{lemma}
The element $d+h$ chosen as above belongs to $\Invar$.
\end{lemma}
\begin{proof}
    An immediate application of Theorem \ref{Knop}.
\end{proof}
Our first main theorem is the following one:
\begin{theorem}\label{irred-non-hol-invar-Weyl}
Let $I=\D(X_f)(d+h)$, $J=I\cap \Invar$. Then $\Invar/J$ is an irreducible $\Invar$-module of Gelfand-Kirillov dimension $2n-1$.
\end{theorem}
\begin{proof}
We have the short exact sequence $0 \rightarrow J \rightarrow \Invar \rightarrow \Invar/J \rightarrow 0$. Since $\D(X_f)=\Invar_\Delta$, $\D(X_f)$ is a flat $\Invar$-module. By \cite[4.12]{Lam2}, $J \otimes_{\Invar} \D(X_f)$ is isomorphic to $J \D(X_f)$, which by \cite[Proposition 2.1.16(iii)]{McConnell} is equal to $I$. Hence we have a short exact sequence
\[ 0 \rightarrow I \rightarrow \D(X_f) \rightarrow (\Invar/J )_\Delta \rightarrow 0.\]
The set $S=\{ \Delta^k \}_{k \geq 0}$ is multiplicatively closed. In a ring of differential operators $\D(Y)$ on a smooth affine variety, every element of $\mathcal{O}(Y)$ acts ad-locally nilpotently. Moreover $\D(Y)$ is a domain. Hence we can apply Theorem \ref{basic-GK-2} to conclude that $\GKdim_{\Invar} \Invar / J = \GKdim_{ \D(X_f)} \D(X_f)/I=2n-1$.
It remains to show that $\Invar / J$ is irreducible, or equivalently, that $J$ is a maximal left ideal. $J$ contains the element $(d+h)$ which, as we just saw in the previous Lemma, belongs to $\Invar$. Suppose there is a proper left ideal $L$ strictly containing $J$. As $\D(X_f)/I$ is a simple $\D(X_f)$-module and localization is an exact functor, $LS^{-1}= \D(X_f)/I$. So $L \cap S \neq \emptyset$, and in particular $L \cap \KK[\bar{E}]$ is a nonzero ideal of $\KK[\bar{E}]$. Let $a$ be a nonzero element from this ideal. Then $[d+h,a]=d(a) \in L \cap \KK[\bar{E}]$. So $L \cap \KK[\bar{E}]$ is a nonzero $d$-invariant ideal of $\KK[\bar{E}]$, and hence, as this ring is $d$-simple, $1 \in L$. This contradicts the fact that $L$ is a proper ideal. So $J$ is indeed maximal, and $\Invar/J$ is our desired module.
\end{proof}
\subsection{A concrete example}
\label{concrete_example}
Let $n=2$ and consider $\Invar$ when $G=S_2$ is the symmetric group on two letters. In this case, writing the Weyl algebra generators as $x_1, x_2, \partial_1, \partial_2$, we have, in the notation of Theorem \ref{EFS}, $\Delta=(x_1-x_2)^2$, $e_1= x_1+x_2$, $e_2=x_1x_2$, and $f=e_1^2-4e_2$.
The algebra $\KK[e_1,e_2]$ is $d$-simple for the derivation $d=\partial_{e_1} +(e_2e_1^2+1)\partial_{e_2}$. If we choose $h=e_2$, then, setting $A=\D(\KK[e_1,e_2])^{S_2}$, the module $A/A(d+h)$ is irreducible non-holonomic of Gelfand-Kirillov dimension $3$ by Theorem \ref{irred-non-hol-invar-Weyl}.
\section{Irreducible non-holonomic modules for rational Cherednik algebras}
In this section we will restrict our attention to complex reflection groups $G$ that are finite Coxeter groups. Let $\h$ be the natural representation of $G$ and $R \subset \h^*$ the set of roots. For each root $\alpha \in R$, write $\alpha^\vee$ for the coroot in $\h$, and denote by $s_\alpha \in \GL(\h)$ the corresponding reflection.
\begin{definition}\cite{EG}
Let $\c: R \rightarrow \C$, $\alpha \mapsto \c_\alpha$ be a function invariant under the conjugation action of $G$. The \textit{rational Cherednik algebra} $H_\c$ is the algebra generated by the vector spaces $\h, \h^*$, and the set $G$, subject to the relations
\[ g \cdot x \cdot g^{-1} =g(x), \forall x \in \h^*, g \in G,\]
\[ g \cdot y \cdot g^{-1} = g(y), \forall y \in \h, g \in G,\]
\[ x_1 \cdot x_2 = x_2 \cdot x_1, \, y_1 \cdot y_2 = y_2 \cdot y_1 , \forall x_1,x_2 \in \h^*, y_1, y_2 \in \h, \]
\[ [y,x] = \langle y, x \rangle - \sum_{\alpha \in R / \{ \pm 1\}} \c_\alpha \langle y, \alpha \rangle \langle \alpha^\vee, x \rangle s_\alpha.\]
\end{definition}
When the algebra $H_\c$ is a simple ring, we say that the parameter $\c$ is \textit{regular} (cf. \cite[Theorem 3.2]{Thompson}). When the image of the function $\c$ lies in $\Z$, we say that the parameter is \textit{integral}.
\begin{theorem}\label{thm-Cherednik}
    If the parameter $\c$ is integral, then it is regular, and the rational Cherednik algebra $H_\c$ is Morita equivalent to $\Invar$.
\end{theorem}
\begin{proof}
    It follows from \cite[Theorem 3.1, Lemma 8.2]{Berest} that $H_\c$ is Morita equivalent to the skew group ring $\Weyl*G$, where $n=\dim \h$. The latter is Morita equivalent to $\Invar$, by Corollary \ref{Morita-invariants-Weyl}.
\end{proof}
When the parameter $\c$ is regular, we define a module over $H_\c$ to be holonomic if its Gelfand-Kirillov dimension equals $n=\dim \h$ (\cite[Proposition 3.7]{Thompson}).
\begin{theorem}\label{iredd-non-hol-Cherednik}
Let $H_\c$ be a rational Cherednik algebra with integral parameter $\c$, and let $n=\dim \h$. Then $H_\c$ has irreducible non-holonomic modules with Gelfand-Kirillov dimension $2n-1$ when $n>1$.
\end{theorem}
\begin{proof}
    $\Invar$ has such modules by Theorem \ref{irred-non-hol-invar-Weyl}. By Theorems \ref{thm-Cherednik} and \ref{GK-is-Morita-for-modules}, we are done.
\end{proof}
We emphasize that, given such an integral parameter $\c$, we can explicitly construct irreducible non-holonomic modules by combining Theorem \ref{irred-non-hol-invar-Weyl} with \cite[Theorem 8.1, Lemma 8.2]{Berest}.
\appendix
\section{Irreducible non-holonomic $\Weyl$-modules with all possible values of $\GKdim$}
We begin by recalling Quillen's celebrated lemma.
\begin{theorem}\cite{Quillen}
    Let $A$ be an associative algebra with a filtration $\mathcal{F}=\{F_i\}_{i\geq 0}$ such that $1\in F_0$ and the associated graded algebra is an affine commutative $\KK$-algebra. Then for every irreducible $A$-module $M$, $\operatorname{End} \, M$ is an algebraic extension of $\KK$. Hence, if $\KK$ is algebraically closed, it equals $\KK$.
\end{theorem}
This is the case for the Weyl algebra with the Bernstein filtration \cite[Theorem 7.3.1]{Coutinho}.
From now on we assume $\KK$ is algebraically closed of characteristic $0$.
\begin{corollary}\label{appb-corol-1}
    For every irreducible $\Weyl$-module $M$, $\operatorname{End} \, M=\KK$.
\end{corollary}
\begin{proof}
    Immediate consequence of Quillen's Lemma.
\end{proof}
Consider $A_n(\KK)$ and $A_m(\KK)$, and let $M$ be an $A_n(\KK)$-module and $N$ an $A_m(\KK)$-module. We have $A_n(\k) \otimes A_m(\k) \simeq A_{n+m}(\k)$, and $M \otimes N$ naturally carries the structure of an $A_{n+m}(\k)$-module (\cite[Chapter 13]{Coutinho}).
Consider $\Weyl$. We have irreducible holonomic modules, such as the polynomial module, and irreducible modules of Gelfand-Kirillov dimension $2n-1$ 
\cite{Stafford}. Let $j$ be a number in $[n+1,2n-2]$. We want to show that there exists an irreducible $\Weyl$-module $M$ with $\GKdim \, M=j$. Let $r=j-n+1$ and $\ell=2n-j-1$. We have an irreducible module $M$ for $A_r(\k)$ with Gelfand-Kirillov dimension $2r-1$, and a holonomic module $N$ for $A_\ell(\k)$. By \cite[Theorem 1]{Bavula1}, $M \otimes N$ is an irreducible $\Weyl$-module, and by \cite[Theorem 13.4.1(1)]{Coutinho}, $\GKdim \, M \otimes N=2r-1+\ell=2j-2n+2-1+2n-j-1=j$.
Hence
\begin{theorem}\label{Olivier}
    For the Weyl algebra $\Weyl$ and any $j \in [n,2n-1]$, there exists an irreducible $\Weyl$-module $M$ with $\GKdim \, M=j$.
\end{theorem}
We finish this paper with an open problem.
\begin{problem}
Let $\k$ be an algebraically closed field of zero characteristic and $G$ any finite group of automorphisms of $\Weyl$. Do there exist irreducible $\Invar$-modules with any value of Gelfand-Kirillov dimension in the interval $[n,2n-1]$?\footnote{Since $\Invar$ is a somewhat commutative algebra, the Gelfand-Kirillov dimension of an irreducible module is always an integer (cf. \cite[Section 8.6]{McConnell}).}
\end{problem}
\section*{Acknowledgements}
The authors would like to thank F. Eshmatov, T. Stafford, E. Zelmanov, and O. Mathieu for discussions about this paper. In particular, we thank O. Mathieu for allowing us to include his proof of Theorem \ref{Olivier}. The second-named author would also like to thank D. Levicovitz for his interest in the problems that led to the creation of this manuscript. FAS was supported by FAPESP grant 2024/14914-9.


\begin{thebibliography}{9}
\bibitem[Arc81]{Archer} J. Archer. Derivations on commutative rings and projective modules over skew polynomial
rings, PhD thesis, Leeds University 1981.
\bibitem[Bav95]{Bavula1} {\sc V. Bavula},
{\em Each Schurian algebra is tensor-simple}. Comm. Algebra 23 (1995), no. 4, 1363–1367.
\bibitem[Bav96]{Bavula2}
{\sc V. Bavula},
{\em Filter dimension of algebras and modules, a simplicity criterion for generalized Weyl algebras},
Comm. Algebra {\bf 24}(6) (1996) 1971--1992.
\bibitem[BH08]{BH}
{\sc V. Bavula and V. Hinchcliffe},
{\em Morita invariance of the filter dimension and of the inequality of Bernstein},
Algebr. Represent. Theory {\bf 11}(5) (2008) 497–504.
\bibitem[BET24]{Bellamy}
{\sc G. Bellamy, P. Etingof, and D. Thompson},
{\em Pull-Back and Push-Forward Functors for Holonomic Modules over Cherednik Algebras}.\\
\href{https://arxiv.org/abs/2402.18210}{arXiv:2402.18210} [math.QA]
\bibitem[BEG03]{Berest}
{\sc Y. Berest, P. Etingof, and V. Ginzburg}, {\em Cherednik algebras and differential operators on quasi-invariants},
Duke Math. J. {\bf 118}(2) (2003) 279--337.
\bibitem[Ber71]{INBernstein}
{\sc I. N. Bernstein},
{\em The analytic continuation of generalized functions with respect to a parameter},
Funct. Anal. Appl 6 (1972) 273--285.
\bibitem[BL88]{BerL} J. Bernstein and V. Lunts. On nonholonomic irreducible D-modules. Invent. Math., 94(2):223–243, 1988.
\bibitem[B79]{Bjork}
{\sc J. E. Bj\"ork},
{\em Rings of differential operators},
North-Holland Math. Library, 21,
North-Holland Publishing Co., Amsterdam-New York, 1979.
\bibitem[Cou95]{Coutinho}
{\sc S. C. Coutinho},
{\em A Primer of Algebraic $D$-Modules},
London Mathematical Society Student Texts, vol. 33, Cambridge University Press, 1995.
\bibitem[Cou99]{Coutinho0}
{\sc S. C. Coutinho},
{\em $d$-Simple rings and simple $\mathscr{D}$-modules},
Math. Proc. Cambridge Philos. Soc. {\bf 125}(3) (1999) 405--415.
\bibitem[Cou03]{Coutinho03}
{\sc S. C. Coutinho}, {\em Non-holonomic simple D-modules over complete intersections}, Proc. Amer. Math. Soc. 131 (2003), no. 1, 83–86.
\bibitem[Cou06]{Coutinho06}
{\sc S. C. Coutinho}, {\em Nonholonomic simple D-modules over projective varieties}, Arch. Math. (Basel) 86 (2006), no. 6, 540–545.
\bibitem[Cou07]{Coutinho07}
{\sc S. C. Coutinho}, {\em Nonholonomic simple D-modules from simple derivations}, Glasg. Math. J. 49 (2007), no. 1, 11–21.
\bibitem[Cou17]{Coutinho17}
{\sc S. C. Coutinho}, {\em Nonholonomic modules over enveloping algebras of semisimple Lie algebras},
J. Algebra 434 (2015), 153–168.
\bibitem[Dum06]{Dumas}
{\sc F. Dumas},
{\em An introduction to noncommutative polynomial invariants}. Lecture notes from the conference ``Homological methods and representations of non-commutative algebras'', Mar del Plata, Argentina March 6--17, 2006.\\
\href{https://lmbp.uca.fr/$\sim$fdumas/recherche.html}{https://lmbp.uca.fr/$\sim$fdumas/recherche.html}
\bibitem[EFOS17]{EFOS}
{\sc F. Eshmatov, V. Futorny, S. Ovsienko, and J. Schwarz},
{\em Noncommutative Noether’s problem
for complex reflection groups},
Proc. Amer. Math. Soc. {\bf 145}(12) (2017) 5043--5052.
\bibitem[EG02]{EG}
{\sc P. Etingof and V. Ginzburg},
{\em Symplectic reflection algebras, Calogero-Moser space, and deformed Harish-Chandra homomorphism},
Invent. Math. {\bf 147}(2) (2002) 243--348.
\bibitem[FS20a]{FS}
{\sc V. Futorny and J. Schwarz},
{\em Noncommutative Noether's problem vs classic Noether's problem},
Math. Z. {\bf 295}(3--4) (2020) 1323--1335.
\bibitem[FS20b]{FS3}
{\sc V. Futorny and J. Schwarz},
{\em Algebras of invariant differential operators},
J. Algebra {\bf 542} (2020) 215--229.
\bibitem[FS21]{FS4}
{\sc V. Futorny and J. Schwarz},
{\em Holonomic modules for rings of invariant differential operators},
Internat. J. Algebra Comput. {\bf 31}(4) (2021) 605--622.
\bibitem[G23]{Gaddis}
{\sc J. Gaddis},
{\em The Weyl algebra and its friends: a survey}, arXiv:2305.01609[math.RA].
\bibitem[GGOR03]{GGOR} {\sc V. Ginzburg, N. Guay, E. Opdam, R. Rouquier},
{\em On the category O for rational Cherednik algebras},
Invent. Math. 154 (2003), no. 3, 617–651.
\bibitem[HTT08]{Hotta}
{\sc R. Hotta, K. Takeuchi, T. Tanisaki},
{\em $D$-modules, perverse sheaves, and representation theory},
Progr. Math., vol. 236, Birkhäuser Boston, Inc., Boston, MA, 2008.
\bibitem[Jan83]{Jantzen}
{\sc J. C. Jantzen},
{\em Einhüllende Algebren halbeinfacher Lie-Algebren},
Einhüllende Algebren halbeinfacher Lie-Algebren,
Ergeb. Math. Grenzgeb. vol. 3, Springer-Verlag, Berlin, 1983.
\bibitem[Kno06]{Knop} {\sc F. Knop }, {\em Graded cofinite rings of differential operators}, Michigan Math. J. 54 (2006), no. 1, 3–23.
\bibitem[Kan01]{Kane} {\sc R. M. Kane},  {\em Reflection groups and invariant theory}, Vol. 5. New York: Springer, 2001.
\bibitem[KL00]{KL}
{\sc G. R. Krause and T. H. Lenegan},
{\em Growth of Algebras and Gelfand-Kirillov Dimension},
Grad. Stud. Math., vol. 22, American Mathematical Society, Providence, RI, 2000.
\bibitem[Lam99]{Lam2}
{\sc T. Y. Lam},
{\em Lectures on modules and rings}
Grad. Texts in Math., vol. 189, Springer-Verlag, New York, 1999.
\bibitem[Lor88]{Lorenz0}
{\sc M. Lorenz},
{\em Gelfand-Kirillov dimension and Poincaré series}, Cuadernos de Algebra, vol. 7, Universidad de Granada, Granada, España, 1988.
\bibitem[Los17]{Losev}
{\sc I. Losev},
{\em Bernstein inequality and holonomic modules},
Adv. Math. {\bf 308} 941--963, 2017.
\bibitem[Lun89]{Lunts}
{\sc V. Lunts},
{\em Algebraic varieties preserved by generic flows},
Duke Math. J. {\bf 58}(3) (1989) 531--554.
\bibitem[MR01]{McConnell}
{\sc J. C. McConnell and J. C. Robson}
{\em Noncommutative Noetherian rings},
Grad. Stud. Math., vol. 30,
American Mathematical Society, Providence, RI, 2001.
\bibitem[Mc82]{McConnell3}
{\sc J. C. McConnell},
{\em Representations of solvable Lie algebras V: On the Gel'fand-Kirillov dimension of simple modules},
J. Algebra {\bf 76}(2) (1982) 489--493.
\bibitem[Mon80]{Montgomery}
{\sc S. Montgomery},
{\em Fixed rings of finite automorphism groups of associative rings},
Lecture Notes in Math., vol. 818, Springer, Berlin, 1980.
\bibitem[MS89]{MStafford}
{\sc J. C. McConnell and J. T. Stafford},
{\em Gel'fand-Kirillov dimension and associated graded modules},
J. Algebra {\bf 125}(1) (1989) 197--214.
\bibitem[MS81]{MS}
{\sc S. Montgomery, L. W. Small},
{\em Fixed rings of noetherian rings},
Bull. London. Math. Soc. {\bf 13}(1) (1981) 33--38.
\bibitem[Qui69]{Quillen}
{\sc D. Quillen},
{\em On the endomorphism ring of a simple module over an enveloping algebra},
Proc. Amer. Math. Soc. {\bf 21} (1969) 171--172.
\bibitem[SKK73]{Sato} M. Sato, T. Kawai, and M. Kashiwara. Microfunctions and pseudo-differential equations. In Hyperfunctions and
pseudo-differential equations (Proc. Conf., Katata, 1971; dedicated to the memory of Andr´e Martineau), volume
Vol. 287 of Lecture Notes in Math., pages 265–529. Springer, Berlin-New York, 1973.
\bibitem[Sta85]{Stafford}
{\sc J. T. Stafford},
{\em Non-holonomic modules over Weyl algebras and enveloping algebras},
Invent. Math. {\bf 79}(3) (1985) 619--638.
\bibitem[Tho18]{Thompson}
{\sc D. Thompson},
{\em Holonomic modules over Cherednik algebras, I.},
J. Algebra 493 (2018) 150--170.
\end{thebibliography}
\end{document}